\documentclass[12pt]{amsart}

\protected\def\ignorethis#1\endignorethis{}
\let\endignorethis\relax
\def\TOCstop{\addtocontents{toc}{\ignorethis}}
\def\TOCstart{\addtocontents{toc}{\endignorethis}}

\usepackage[all]{xy}
\usepackage{graphicx}

\usepackage{pxfonts}

\usepackage[dvipsnames]{xcolor}
\colorlet{Jonathan}{ForestGreen}
\colorlet{Sandro}{RawSienna}
\usepackage{stackengine}
\usepackage[normalem]{ulem}
\usepackage[colorlinks]{hyperref}
\usepackage[colorinlistoftodos]{todonotes}
\usepackage[margin=3.5cm]{geometry}
\usepackage{amssymb}
\usepackage{amsmath}
\usepackage{amsthm}
\usepackage{mathrsfs}
\usepackage{amsfonts}
\usepackage{comment}
\usepackage{mathtools}
\usepackage{enumerate}
\usepackage{tikz-cd}
\usepackage{latexsym} 
\newtheorem{thm}{Theorem}[section]
\newtheorem{definition}[thm]{Definition}
\newtheorem{prop}[thm]{Proposition}
\newtheorem{lemma}[thm]{Lemma}
\newtheorem{claim}[thm]{Claim}

\theoremstyle{remark}
\newtheorem*{rmk}{Remark}
\newtheorem{ex}[thm]{Example}

\newcommand{\Vect}{{\mathrm{Vect^{f}}}}
\newcommand{\id}{{\mathrm{id}}}

\newcommand{\cala}{{\mathcal A}}
\newcommand{\calb}{{\mathcal B}}
\newcommand{\calc}{{\mathcal C}}
\newcommand{\cald}{{\mathcal D}}

\newcommand{\calf}{{\mathcal F}}
\newcommand{\calg}{{\mathcal G}}

\newcommand{\calm}{{\mathcal M}}
\newcommand{\caln}{{\mathcal N}}

\newcommand{\fieldk}{{\mathbf k}}

\newcommand{\arrows}{\,\lower1pt\hbox{$\longrightarrow$}\hskip-.24in\raise2pt
             \hbox{$\longrightarrow$}\,}

\newcommand{\Mod}{\mathrm{Mod}}
\newcommand{\Alg}{\mathrm{Alg}_{2}}
\newcommand{\Rep}{\mathrm{Rep}}
\newcommand{\Endo}{\mathrm{End}}
\newcommand{\Vectt}{\mathrm{Vect}}
\newcommand{\lincat}{\mathrm{LinCat}}
\newcommand{\unit}{\mathbf{1}}
\newcommand{\KV}{\mathrm{KV}_{\fieldk}}
\newcommand{\Cat}{\mathrm{Cat}}
\newcommand{\homm}{\mathrm{hom}}
\newcommand{\et}{\mathfrak{y}}
\newcommand{\ii}{\mathfrak{i}}
\usepackage{url}

\newcommand{\xc}[1]{}

\usepackage{tikz}
\usepackage{tikz-cd}
\usetikzlibrary{math}
\usetikzlibrary{calc}

\usepackage{tikz}
\usetikzlibrary{calc, intersections}

\AtBeginDocument{%
   \def\MR#1{}
}

\title{Morita bicategories of algebras and duality involutions}
\author{Jonathan Lorand}
\author{Alessandro Valentino}
\email{jonathan.lorand@math.uzh.ch}
\email{alessandro.valentino@math.uzh.ch}
\address{Institute of Mathematics, University of Zurich, Winterthurerstrasse 190, CH-8057, Zurich, Switzerland}

\begin{document}

\begin{abstract}
The notion of a weak duality involution on a bicategory was recently introduced by Shulman in \cite{shulman}. We construct a weak duality involution on the fully dualisable part of $\Alg$, the Morita bicategory of finite-dimensional $\fieldk$-algebras. The 2-category $\KV$ of Kapranov-Voevodsky $\fieldk$-vector spaces may be equipped with a canonical strict duality involution. We show that the pseudofunctor $\Rep:\Alg^{fd}\to\KV$ sending an algebra to its category of finite-dimensional modules may be canonically equipped with the structure of a duality pseudofunctor. Thus $\Rep$ is a strictification in the sense of Shulman's strictification theorem for bicategories with a weak duality involution.

Finally, we present a general setting for duality involutions on the Morita bicategory of algebras in a semisimple symmetric finite tensor category.     
\end{abstract}
\maketitle
\begin{flushright}
{\tiny
\parbox{7cm}{
\emph{``I learned to recognise the thorough and primitive duality of man; I saw that, of the two natures that contended in the field of my consciousness, even if I could rightly be said to be either, it was only because I was radically both.''}\\[2mm]
\textbf{ The Strange Case of Dr. Jekyll and Mr. Hyde}\\ 
R. Stevenson}}
\end{flushright}
\tableofcontents
\section{Introduction}
It is well-known that the category of (right) modules $\Mod_{A}$ over an algebra $A$ provides important information about $A$ itself. Indeed, much of modern algebra is concerned with the study of categories of representations and their structures. A classical notion of equivalence between algebras is \emph{Morita equivalence}: introduced in \cite{morita}, two algebras $A$ and $B$ are Morita equivalent if their categories of right modules $\Mod_{A}$ and $\Mod_{B}$ are equivalent. An elegant reformulation of Morita equivalence between algebras can be obtained via the language of bicategories. Briefly, by regarding algebras as objects of a \emph{Morita bicategory}\footnote{See the discussion of the name at 
\begin{center} \texttt{https://mathoverflow.net/questions/225701/reference-request-morita-bicategory} . \end{center}}, Morita equivalence corresponds to the notion of equivalence internal to a bicategory. Since the Morita bicategory is a convenient (higher) categorical environment where algebras and their equivalences live, it is natural to investigate the various structures that such a bicategory supports.

In this short paper we investigate one particular structure recently introduced in \cite{shulman}, namely a \emph{weak duality involution} on a bicategory, which axiomatises and generalises the operation of ``taking the opposite category'', together with opposite functors and natural transformations. In a precise sense, this can be regarded as a categorification of taking the dual of an object in a rigid monoidal category. Concretely, we will construct a canonical weak duality involution on the \emph{fully dualisable} sub-bicategory of the Morita bicategory $\Alg$ of \emph{finite-dimensional} algebras. The full dualisability condition, which we explain in the paper, can be morally regarded as a finiteness condition on objects and 1-morphisms of a bicategory. The appearence of fully dualisable bicategories opens an interesting relation to the study of framed fully extended 2d topological quantum field theories, as in \cite{lurieHigher, schompries}. More precisely, the \emph{core} of the fully dualisable part of $\Alg$ corresponds to the symmetric monoidal bifunctors from the framed two-dimensional bordism category $\mathrm{Bord}^{fr}_{2}$ to $\Alg$ itself. It is then natural to expect that $\mathrm{Bord}^{fr}_{2}$ comes equipped with a weak duality involution of geometric origin. Though one of the hidden motivations behind the present work, we will leave this line of research to future developments.

After quickly discussing how the 2-category $\KV$ of \emph{Kapranov-Voevodsky} vector spaces corresponds to the fully dualisable part of $\lincat_{\fieldk}$, we show that $\KV$ can be canonically equipped with a \emph{strict} duality involution. We then consider the bifunctor $\Rep$ which sends an algebra to its category of representation. We prove that $\Rep$ can be canonically equipped with all the necessary data of a \emph{duality pseudofunctor}. Since $\Rep$ is an equivalence of bicategories, this can be regarded as an instance of the strictification theorem proven in \cite{shulman}, which states that any bicategory with weak duality involution is biequivalent to a 2-category with strict duality involution via a duality pseudofunctor.

The constructions presented in Section \ref{section:alg} and \ref{section:rep} are structural enough to allow for a generalisation. In the last part of the paper we consider the case of algebras in a symmetric semisimple finite tensor category $\calc$, and their Morita bicategory $\Alg(\calc)$. We identify the target of the representation bifunctor $\Rep^{\calc}$ as the 2-category $\Mod^{ss}(\calc)$ of semisimple \emph{module categories} over $\calc$, which we briefly recall in the paper. After equipping $\Mod^{ss}(\calc)$ with a weak duality involution, we argue that $\Rep^{\calc}$ can be made into a duality pseudofuntor.

The paper is organized as follows.

In Section \ref{section: moritabicat} we briefly recall some basic aspects of the Morita bicategory of finite dimensional algebras, and we fix some notation regarding modules over an algebra.

In Section \ref{section:finite} we discuss finite linear categories and illustrate some properties of fully dualisable bicategories. We also discuss Kapranov-Voevodsky vector spaces and the representation bifunctor.

In Section \ref{section:duality} we review weak duality involutions on bifunctors and duality pseudofunctors.

In Section \ref{section:alg} we construct a weak duality involution on the fully dualisable sub-bicategory $\Alg^{fd}$ of the Morita bicategory of finite dimensional algebras. This is the content of Theorem \ref{circ is duality inv}.

In Section \ref{section:rep},  in Theorem \ref{thm:rep}, we show that the representation bifunctor $\Rep:\Alg^{fd}\to \KV$ can be canonically equipped with the structure of a duality pseudofunctor, providing a strictification biequivalence.

Finally, in Section \ref{section:general} we briefly describe a generalisation of the results obtained in the previous sections. In particular, we consider module categories and argue that they come equipped with a canonical weak duality involution. We then state a claim concerning the representation pseudofunctor $\Rep^{\calc}:\Alg^{fd}(\calc)\to\Mod^{ss}(\calc)$.

In Appendix \ref{appendix:modules}, we provide some background material concerning modules over finite-dimensional algebras, while in Appendix \ref{appendix:proof} we give the details of the proof of Theorem \ref{circ is duality inv}.

Throughout the paper, we assume the reader to be familiar with the language of bicategories and associated higher categorical constructions. Also, we always assume the field $\fieldk$ has characteristic 0 and is algebraically closed.
\TOCstop
\section*{Acknowledgements}
\TOCstart
The authors would like to thank Mike Shulman and Daniel Tubbenhauer for useful comments and discussions. J.L. and A.V. acknowledge support from the NCCR SwissMAP and partial support from SNF Grant \text{No.} 200020 172498/1, both funded by the Swiss National Science Foundation;  A.V. also acknowledges support from the COST Action MP1405 QSPACE, funded by COST (European Cooperation in Science and Technology).
\section{The Morita bicategory of algebras}\label{section: moritabicat}
In this section we briefly review some aspects of the Morita bicategory of algebras, and some standard notation regarding modules and bimodule over finite-dimensional algebras; see Appendix \ref{appendix:modules} for more details. We refer the reader to \cite{benabou, leinsterbicat} for the terminology and details concerning bicategories and their functors.
\begin{definition}
The \emph{Morita bicategory of algebras} $\mathrm{Alg}_{2}$ is the bicategory where
\begin{itemize}
\item the objects are finite-dimensional $\fieldk$-algebras;
\item the 1-morphisms from $A$ to $B$ are finite-dimensional $\fieldk$-vector spaces which are $(A,B)$-bimodules; and
\item the 2-morphisms are intertwiners of bimodules, \text{i.e.} $\fieldk$-linear maps of bimodules which are compatible with the respective left and right actions of $\fieldk$-algebras.
\end{itemize}
\end{definition}
Throughout the paper, the terms ``algebra'' and ``bimodule'' will always refer to the sort appearing in the definition of $\mathrm{Alg}_{2}$; similarly for ``right modules'', \text{etc.}. Moreover, if $M$ is an $(A,B)$-bimodule, we indicate this by writing ${}_{A}M_B$, though at times we will drop the subscripts. In the following, we schematically recall some basic features of $\mathrm{Alg}_{2}$ and related notions which will be useful in later sections of the paper.

The composition operations for 1- and 2-morphisms in $\Alg$ are defined as follows\footnote{We work under the assumption that representatives for tensor products have been fixed.}
\begin{itemize}
\item \emph{composition of 1-morphisms}: for $_{A}M_{B}$ and $_{B}N_{C}$ bimodules, their composition is defined as 
\begin{equation}
N\circ{M}:={}_{A}(M\otimes_{B}N)_{C}
\end{equation}
\item \emph{horizontal composition of 2-morphisms}: for $f: {}_{A}M_{B}\to {_{A}N_{B}}$ and $g: {}_{B}M{'}_{C}\to {_{B}N{'}_{C}}$ intertwiners of bimodules, their horizontal composition is defined as
\begin{equation}
g\bullet^{h}{f}:=f\otimes_{B}{g}
\end{equation}
\item \emph{vertical composition of 2-morphisms}: for $f: {}_{A}M_{B}\to {}_{A}N_{B}$ and $g: {}_{A}N_{B}\to {}_{A}P_{B}$ intertwiners, their vertical composition is defined as
\begin{equation}
g\bullet^{v}f:=g\circ{f}
\end{equation}
\end{itemize}
We refer to Appendix \ref{appendix:modules} for details on the constructions above.

The coherence data for the bicategory $\Alg$ arise as follows
\begin{itemize}
\item \emph{associators}: for $_{A}M_{B}$, $_{B}N_{C}$ and $_{C}P_{D}$, the associator isomorphism
\begin{equation}
\alpha_{M,N,P}:P\circ(N\circ{M})\xrightarrow{\simeq}(P\circ N)\circ{M}
\end{equation}
is given by the canonical isomorphism $(M\otimes_{B}N)\otimes_{C}P\simeq M\otimes_{B}(N\otimes_{C}P)$ of tensor products of bimodules;
\item \emph{unitors}: for any algebra $A$, the unit 1-morphism $1_{A}:A\to{A}$ is given by $A$ itself regarded as a $(A,A)$-bimodule; for any $(A,B)$-bimodule $M$, the left and right unitor isomorphisms 
\begin{equation}
1_{B}\circ{M}\simeq{M}\simeq{M}\circ 1_{A}
\end{equation}
are given by the canonical isomorphisms $M\otimes_{B}B\simeq{M}\simeq A\otimes_{A}M$.
\end{itemize}
The isomorphisms above satisfy the compatibility diagrams for the coherence data of a bicategory.
\begin{rmk}
Our notation for the Morita bicategory of algebras differs from the one used in \cite{lurieTQFT}.
\end{rmk}
Recall the following 
\begin{definition}
Let $\calb$ be a bicategory. A 1-morphism $f:x\to{y}$ is called an \emph{equivalence} if there exists a 1-morphism $g:y\to{x}$, and invertible 2-morphisms
\begin{equation}
\id_{x}\simeq g\circ{f},\quad f\circ{g}\simeq\id_{y}
\end{equation}
\end{definition}
Two objects $x$ and $y$ in a bicategory $\calb$ are called \emph{equivalent} if there exists an equivalence between $x$ and $y$.

The following is a well-known result.
\begin{prop}
Two algebras are equivalent as objects in $\Alg$ if and only if they are Morita equivalent.
\end{prop}

The following notation will be used (hopefully) consistently throughout the paper.

If ${}_{A}M_B$ and ${}_{C}M{'}_{D}$ are bimodules, we denote with $\homm{_\fieldk}(M,M')$ the vector space of $\fieldk$-linear maps from $M$ to $M'$. Given bimodules ${}_{A}M_B$ and ${}_{D}N_B$, $\homm_{B}(M,N)$ denotes the set of right $B$-module morphisms, namely elements of $\hom{_\fieldk}(M,N)$ which, additionally, are compatible with the right $B$-action. We avoid completely the analogous notion for left modules, so that our notation for morphisms of right-modules is unambiguous. 

Given bimodules ${}_{A}M_B$ and ${}_{C}N_B$, the vector space $\homm_{B}(M,N)$ may naturally be equipped with a left $C$-action and a right $A$-action. Indeed, these are defined as 
\begin{align}
\begin{split}
\homm_{B}(M,N)\times{A}& \to \homm_{B}(M,N)\\
 (f,a) & \mapsto fa: x \mapsto f(ax)
 \end{split}
\end{align}
and
\begin{align}
\begin{split}
C\times\homm_{B}(M,N)& \mapsto \homm_{B}(M,N)\\
 (c,f) & \mapsto cf: x \mapsto cf(x),
 \end{split}
\end{align}
respectively. We write ${}_{C}\homm_{B}(M,N)_{A}$ to indicate this bimodule structure, and we always assume these left and right actions unless otherwise indicated.

Recall that from any algebra $A$ we obtain an opposite algebra $A^{op}$. This is the $\fieldk$-algebra which has the same underlying vector space as $A$, and the same unit, but where the multiplication is inverted. For notational ease, we denote multiplication using juxtaposition when it is clear which algebra $A$ is at play; the notation $\star$ indicates when we are multiplying in the opposite algebra. 

Finally, recall that any left $A$-module can be viewed as right $A^{op}$-module. Indeed, for $M$ a left $A$-module we can consider the following right $A^{op}$-action 
\begin{align}
\begin{split}
M\times{A^{op}} & \to M\\
(m,a) &\mapsto a\cdot{m}
\end{split}
\end{align}
It is readily checked that this does indeed define a right action. In a similar fashion, bimodules ${}_{A}M_B$ may be viewed as bimodules ${}_{B^{op}}M_{A^{op}}$. We will make this kind of switch tacitly when no confusion is to be feared. 
\section{Finite categories and dualisability}\label{section:finite}
In this section we provide a review of well-known material, mainly following \cite{tenscat, finitetensor, DSPS}. This will be useful both to give a precise characterisation of the bifunctor $\Rep$, and in view of the general setting of Section \ref{section:general}.
\subsection{Finite linear categories} 
For $\fieldk$ a fixed ground field, recall that a \emph{linear category} is an abelian category enriched over $\Vectt_{\fieldk}$, the symmetric monoidal category of $\fieldk$-vector spaces, not necessarily finite-dimensional. A \emph{linear functor} is an additive functor which is also a functor of $\Vectt_{\fieldk}$-enriched categories.
\begin{definition}A linear category $\calc$ is called \emph{finite} if:
\begin{itemize}
\item $\calc$ has finite-dimensional vector spaces as spaces of morphisms;
\item every object of $\calc$ has finite length;
\item $\calc$ has enough projectives; and
\item there are finitely many isomorphism classes of simple objects.
\end{itemize}
\end{definition}
\begin{ex}
An example of a finite linear category is $\Vect_{\fieldk}$, the category of finite-dimensional $\fieldk$-vector spaces.
\end{ex}
The following proposition is important for recognizing finite linear categories.
\begin{prop} A linear category $\calc$ is finite if and only if it is equivalent to the category $\Mod_{A}$ of finite-dimensional (right) modules over a finite-dimensional $\fieldk$-algebra $A$.
\end{prop}
Recall that an additive functor between abelian categories is called \emph{left exact} if it sends left exact sequences to left exact sequences. The notion of a right exact functor is similar.

We can consider now the 2-category of finite linear categories.
\begin{definition} For a fixed ground field $\fieldk$, the 2-category $\lincat_{\fieldk}$ has:
\begin{itemize}
\item finite linear categories as objects;
\item right exact functors as 1-morphisms; and
\item natural transformations as 2-morphisms.
\end{itemize}
\end{definition}
We can consider $\lincat_{\fieldk}$ as a linearization of $\Alg$ via the \emph{representation bifunctor}. More precisely, consider the bifunctor
\begin{equation}
\Rep:\Alg\to\lincat_{\fieldk}
\end{equation}
defined as follows:
\begin{itemize}
\item to a finite-dimensional algebra $A$ it assigns the category $\Mod_{A}$;
\item to a finite-dimensional $(A,B)$-bimodule $_{A}M_{B}$ it assigns the right exact functor $(-)\otimes_{A}M_{B}:\Mod_{A}\to\Mod_{B}$; and
\item to an intertwiner $f$ between $_{A}M_{B}$ and $_{A}N_{B}$ it assigns the corresponding natural transformation between $(-)\otimes_{A}M_{B}$ and $(-)\otimes_{A}N_{B}$. 
\end{itemize}
Notice that the various isomorphisms needed to make $\Rep$ into a bifunctor arise canonically from the properties of the tensor product of modules.

As pointed out in \cite{bartlettschom2}, following \cite{DSPS,schompries} one obtains 
\begin{prop}
The bifunctor $\Rep$ is an equivalence of bicategories
\end{prop}
\begin{rmk}
The bifunctor $\Rep$ is actually an equivalence of symmetric monoidal bicategories. See \cite{schompries} for details on symmetric monoidal structures on bicategories.
\end{rmk}
\subsection{Full dualisability}We now recall some basic notions concerning adjoints for 1-morphisms in bicategories and full-dualisability.

Let $\calb$ be a bicategory.
\begin{definition}
A 1-morphism $f:x\to{y}$ in $\calb$ \emph{admits a right adjoint} if there exists a 1-morphism $g:y\to{x}$, and 2-morphisms $\epsilon:f\circ g\to\id_{y}$ and $\eta:\id_{x}\to{g\circ{f}}$ satisfying the triangle identities. 
\end{definition}
Similarly, we have the notion of a left adjoint of a 1-morphism.

An $\emph{adjunction}$ $f\dashv{g}$ is a collection $(f,g,\epsilon,\eta)$ such that $g$ is a right adjoint to $f$ via $\epsilon$ and $\eta$. We say that $f\dashv{g}$ is an \emph{adjoint equivalence} if $\epsilon$ and $\eta$ are invertible 2-morphisms.

The following theorem will be useful in later sections.
\begin{thm}[\cite{gurski}]\label{thm:adjointequiv}
Let $\calb$ be a bicategory, and let $f$ be an equivalence in $\calb$. Then $f$ is part of an adjoint equivalence $f\dashv{g}$.
\end{thm} 
\begin{rmk}
As remarked in \cite{gurski}, the theorem above guarantees something stronger than the existence of an adjoint equivalence. Indeed, given an equivalence $f:x\to{y}$ in $\calb$, a (pseudo) inverse $g$ and a 2-isomorphism $\alpha:f\circ{g}\simeq\id_{y}$, there exists a \emph{unique} adjoint equivalence $(f,g,\epsilon,\eta)$ with $\epsilon = \alpha$.
\end{rmk}
\begin{ex}
Let $\calc$ be a monoidal category. If we regard $\calc$ as a bicategory with a single object, then a 1-morphism $x$ admits a right (resp. left) adjoint if and only if $x$ admits a left (resp. right) dual as an object in $\calc$. 
\end{ex}
\begin{definition}
A bicategory $\calb$ is said to \emph{admit duals for 1-morphisms} if any 1-morphism admits a right and a left adjoint. 
\end{definition}
In the following we recall the definition of duals in symmetric monoidal bicategories; see \cite{schompries} for details.
\begin{definition}
Let $(\calb, \otimes,\unit)$ be a symmetric monoidal bicategory. An object $x\in\calb$ is \emph{dualisable} if there exists $x^{*}\in\calb$ and 1-morphisms $e:x\otimes x^{*}\to\unit$ and $c:\unit \to x^{*}\otimes{x}$ sastisfying the zig-zag identities up to 2-isomorphisms.
\end{definition} 
\begin{rmk}The statement regarding the zig-zag equations means that for any dualisable object $x\in\calb$ there are isomorphisms
\begin{align}
\begin{split}
(e\otimes\id_{x})\circ(\id_{x}\otimes{c})&\simeq\id_{x}\\
(\id_{x^{*}}\otimes e)\circ(c\otimes{\id_{x^{*}}})&\simeq\id_{x^{*}} .
\end{split}
\end{align}
See for instance \cite{lurieTQFT}.
\end{rmk}
\begin{definition}
A symmetric monoidal bicategory $\calb$ is said to \emph{admit duals for objects} if any object is dualisable.
\end{definition}
We can combine the two requests on a bicategory via the following
\begin{definition}
A symmetric monoidal bicategory $\calb$ is said to be \emph{fully dualisable} if it admits duals for objects and 1-morphisms. 
\end{definition}
Given a symmetric monoidal bicategory $\calb$, we denote with $\calb^{fd}$ the maximal sub-bicategory of $\calb$ which is fully dualisable. An object in $\calb^{fd}$ is called \emph{fully dualisable}.

We now discuss the fully dualisable part of the (symmetric monoidal\footnote{We will not indulge in the gory details of their symmetric monoidal products.}) bicategories of interest for the present work, namely $\Alg$ and $\lincat_{\fieldk}$; our main reference will be Appendix A of \cite{bartlettschom2}.

From \cite{davydovich,schompries} it follows that $\Alg^{fd}$ corresponds to the full sub-bicategory of $\Alg$ spanned by semi-simple\footnote{Semi-simplicity arises from the assumption that $\fieldk$ has characteristic 0; \emph{separability} is a suitable notion otherwise.} (finite-dimensional) $\fieldk$-algebras. Note that any finite-dimensional module over a semi-simple finite-dimensional algebra is automatically projective; see Appendix \ref{appendix:modules}.

To discuss the fully dualisable part of $\lincat_{\fieldk}$, we need first the following
\begin{definition}A Kapranov-Voevodsky (KV) vector space is a finite linear category which is semi-simple and equivalent to $\Vectt_{\fieldk}^{n}$ for some $n$.
\end{definition}
From \cite{bartlettschom1} we have that $\lincat_{\fieldk}^{fd}$ is the full\footnote{Note that any right exact functor between semi-simple abelian categories is automatically left exact.} sub-2-category of $\lincat_{\fieldk}$ spanned by KV-vector spaces. For simplicity we use $\KV$ to denote $\lincat_{\fieldk}^{fd}$.

From the fact that $\Rep$ is a symmetric monoidal biequivalence one has
\begin{prop}
The bifunctor $\Rep$ induces by restriction an equivalence of bicategories between $\Alg^{fd}$ and $\KV$.
\end{prop}
The proposition above is guaranteed by the fact that any symmetric monoidal bifunctor $\cala^{fd}\to\calb$ factors uniquely through $\calb^{fd}$, and by the maximality property of fully dualisable subcategories.
\begin{rmk}
The definition of a Kapranov-Voevodsky vector space provided above is slightly different from that in \cite{kv}; see Section \ref{section:general} for comments.
\end{rmk}
\section{Duality involutions and functors}\label{section:duality}
In this section we briefly recall the notion of a duality involution on a bicategory as introduced in \cite{shulman}, which we also use as the main source for the details needed in the present section.

In the following, $\cala$ and $\calb$ denote bicategories. 
\begin{definition}
Let $\cala$ be a bicategory. Then $\cala^{co}$ denotes the bicategory with the same objects as $\cala$, and
\begin{equation}
\cala^{co}(x,y):=\cala(x,y)^{op},\quad \forall x,y\in\cala
\end{equation} .
\end{definition}
In other words, $\cala^{co}$ is the bicategory obtained from $\cala$ by reversing 2-morphisms.

One has that any bifunctor $F:\cala\to\calb$ induces a bifunctor $F^{co}:\cala^{co}\to\calb^{co}$, defined in the obvious way, and similarly for natural transformations and their modifications\footnote{Beware of the fact that $\theta^{co}:\gamma^{co}\to{\eta}^{co}$ for a modification $\theta:\eta\to\gamma$ .}.

\begin{definition}\label{def:duality}
A \emph{weak duality involution} on $\cala$ is the following collection of data:
\begin{itemize}
\item a pseudofunctor $(-)^{\circ}:\cala^{co}\to\cala$;
\item a pseudonatural adjoint equivalence
\begin{displaymath}
\begin{tikzcd}
\cala \arrow[rr, equal, ""{name=U, below}] \arrow[rd, "((-)^{\circ})^{co}"']& & \cala\\
& \cala^{co} \arrow[Rightarrow, from=U, "\et"] \arrow[ru, "(-)^{\circ}"']&
\end{tikzcd};
\end{displaymath}
and
\item an invertible modification $\zeta$, given in components by 
\begin{equation}
\zeta_{x}:\et_{x^{\circ}}\xrightarrow{\simeq}(\et_{x})^{\circ},\quad \forall x \in \cala
\end{equation}
satisfying a compatibility diagram; see \cite{shulman}.
\end{itemize}
\end{definition}
In the case in which $(-)^{\circ}$ is a strict\footnote{This requires $\cala$ to be a strict bicategory.} bifunctor, $\et$ is a strict binatural isomorphism, and $\zeta$ is the identity modification, we have a \emph{strong duality involution} on $\cala$. Moreover, if in the case before $\et$ is the identity as well, we have a \emph{strict duality involution} on $\cala$.

A prototypical example of a (strict) duality involution is provided by taking the opposite category. Indeed, denote with $\Cat$ the 2-category of small categories, and consider the following 2-functor
\begin{equation}
(-)^{op}:\Cat^{co} \to\Cat\\
\end{equation}
defined as follows:
\begin{itemize}
\item to a category $\calc$ it assigns the opposite category $\calc^{op}$;
\item to a functor $F$ between $\calc$ and $\cald$ it assigns the opposite functor $F^{op}$ between $\calc^{op}$ and $\cald^{op}$; and
\item to a natural transformation $\epsilon$ between $F$ and $G$ it assigns the opposite natural transformation $\epsilon^{op}$ between $G^{op}$ and $F^{op}$.
\end{itemize}

Note that $(-)^{op}$ is defined on $\Cat^{co}$ since taking the opposite of a natural transformation between functors changes its direction.

Since taking the opposite twice is strictly the identity operation, we can choose the components of $\et$ to be the identity 1-cells; moreover, we can choose the 2-cells witnessing the naturality to be identity 2-cells as well. Finally, if we choose the components of $\zeta$ to be identity 2-cells also, one can easily show that the above data satisfy the required compatibility diagram. Hence, we have that $(-)^{op}$ canonically provides a strict duality involution on $\Cat$.

It is interesting to notice at this point that taking the opposite category does \emph{not} provide a strict duality involution on $\lincat_{\fieldk}$. Indeed, though the opposite category of a finite linear category is again a finite linear category, the opposite of a right exact functor is \emph{left} exact. On the other hand, $(-)^{op}$ does provide a strict duality involution on $\KV$, since morphisms between KV-vector spaces are exact functors.
\begin{definition}\label{def:dualityfunctor}
Let $\cala$ and $\calb$ be bicategories equipped with a weak duality involution. A \emph{duality pseudofunctor} between $\cala$ and $\calb$ is a pseudofunctor $F:\cala\to\calb$ equipped with
\begin{itemize}
\item a pseudonatural adjoint equivalence
\begin{displaymath}
\begin{tikzcd}
\cala^{co} \arrow[d, "(-)^{\circ}"']\arrow[r, "F^{co}"] & \calb^{co} \arrow[d, "(-)^{\circ}"] \arrow[shorten <=10pt,shorten >=10pt, Rightarrow, to=2-1, "\ii"] \\
\cala  \arrow[r, "F"'] & \calb 
\end{tikzcd};
\end{displaymath}
and
\item an invertible modification $\theta$ whose components are 2-morphisms in $\calb$ of the following form
\begin{displaymath}
\begin{tikzcd}
(Fx)^{\circ\circ}\arrow[r, "(\ii_{x})^{\circ}" {name=U}] & (F(x^{\circ}))^{\circ}\arrow[d, "\ii_{x^{\circ}}"]\\
Fx \arrow[u, "\et_{Fx}" ]\arrow[r, "F(\et_{x})"'{name=D}] & F(x^{\circ\circ}) \arrow[shorten <=10pt,shorten >=10pt, Rightarrow, from=U, to=D, "\theta_{x}"]
\end{tikzcd}
\quad,\:\forall x\in\cala ,
\end{displaymath}
satisfying a compatibility diagram involving $\zeta$, $\et$ and $\ii$; see \cite{shulman}.
\end{itemize}
\end{definition}
Similar to the case of a weak duality involution, we have the notion of a \emph{strong duality pseudofunctor} and \emph{strict duality pseudofunctor}.

The notion of a duality pseudofunctor allows to formulate the following theorem, which is one the main results in \cite{shulman}.
\begin{thm} Let $\cala$ be a bicategory with a weak duality involution. Then there exists a 2-category $\cala'$ with a strict duality involution and a duality pseudofunctor $\cala\to\cala'$ that is a biequivalence.
\end{thm}
The theorem above is essentially a coherence theorem for bicategories with duality involutions, which ensures that there is no loss in generality in considering only strict duality involutions. In \cite{shulman}, the theorem is proven by using the theory of 2-monads and representable multicategories.

In the following, which constitutes the main result of the present work, we provide a concrete illustration of the above theorem involving naturally occurring bicategories with duality involutions, namely $\Alg^{fd}$ and $\KV$ considered in Section \ref{section:finite}.
\begin{rmk} Morita bicategories of algebras have a natural generalisation to the case of $(\infty,1)$-categories \cite{haugseng}. Morever, the constructions can be reiterated to higher algebraic structures, such as $E_{n}$-algebras \cite{haugseng, gwillscheim}. It would therefore be interesting to properly develop a theory of duality involutions adapted to the $\infty$-world. 
\end{rmk}
\section{A duality involution on $\mathrm{Alg}^{fd}_{2}$}\label{section:alg}
In this section we explicitely construct a weak duality involution on the Morita bicategory $\Alg^{fd}$. In the next section we will then prove that such a weak duality involution \emph{strictifies} to the duality involution on $\KV$ described in Section \ref{section:duality}.

Consider the bifunctor
\begin{equation}
(-)^{\circ}:(\Alg^{fd})^{co}\to\Alg^{fd}
\end{equation}
defined as follows:
\begin{itemize}
\item to an object, i.e an algebra $A$ it assigns $A^\circ := A^{op}$, the opposite algebra;
\item to a 1-morphism, i.e. a bimodule $_{A}M_{B}$ it assigns $({}_{A}M_B)^\circ := {}_{A^{op}}(\homm_B(M,B))_{B^{op}}$\footnote{Here we are taking the $(B,A)$-bimodule $\homm_B(M,B)$ and viewing it as an as $(A^{op},B^{op})$-bimodule. As mentioned above, we will henceforth perform this operation tacitly without further remark.}; and
\item to a 2-morphism, i.e. an intertwiner $f$ it assigns $f^\circ := f^*$.
\end{itemize}
In the above definition, $f^*$ denotes the adjoint map, namely it is given by the operation ``precompose with $f$''.

For $(-)^\circ$ to be a bifunctor, we need to specify invertible $2$-morphisms in $\Alg^{fd}$
\begin{equation}\label{eq:coe1}
({}_{A}M \otimes_B N_C)^\circ \Rightarrow ({}_{A}M_B)^\circ \otimes_{B^{op}} ({}_{B}N_C)^\circ
\end{equation}
and
\begin{equation}\label{eq:coe2}
1_A^\circ \Rightarrow 1_{A^\circ} .
\end{equation}
satisfying compatibility diagrams.

First, notice that we have the following isomorphisms of $(C,A)$-bimodules
\begin{align}
\begin{split}
\homm_{C}(M\otimes_{B}N,C)&\simeq \homm_{B}(M,\homm_{C}(N,C))\\
& \simeq \hom_{C}(N,C)\otimes_{B}\hom_{B}(M,B)
\end{split}
\end{align}
which is natural in $M$ and $N$; see Appendix \ref{appendix:modules} for details. Notice now that for arbitrary bimodules $_{A}M_{B},\: _{B}N_{C}$ and $_{A}S_{C}$, any morphism $_{A}M\otimes_{B}N_{C} \to {}_{A}S_{C}$ can be regarded as a morphism $_{C^{op}}N\otimes_{B^{op}}M_{A^{op}} \to {}_{C^{op}}S_{A^{op}}$. We then get the isomorphism in (\ref{eq:coe1}).

Consider now the natural isomorphism of \emph{algebras}
\begin{equation}
\homm_A(A,A)\xrightarrow{\simeq}A^{op}
\end{equation}
given by
\begin{equation}
f \longmapsto f(1) .
\end{equation}
It is straightforward to check that the above isomorphism is an isomorphism of $(A,A)$-bimodules, where we canonically regard $A^{op}$ equipped with the $(A^{op})^{op}=A$ left and right actions. By regarding them both as $(A^{op},A^{op})$-bimodules we obtain the isomorphism (\ref{eq:coe2}).

Notice that the required naturality with respect to 2-morphisms of the isomorphism $(\ref{eq:coe1})$ and $(\ref{eq:coe2})$ is guaranteed by the naturality of the various isomorphisms of bimodules involved.
\begin{rmk}
The isomorphisms above, in particular (\ref{eq:coe1}), are available because the objects of $\Alg^{fd}$ are finite-dimensional \emph{semisimple} algebras, and hence all bimodules are projective.
\end{rmk}
We now proceed to construct the pseudonatural adjoint equivalence $\et$ and the modification $\zeta$ required in definition \ref{def:duality}.

Following \cite{gurski}, for $\et$ it is enough to give the data of a pseudonatural transformation of bifunctors
\begin{equation}
1_{\mathrm{Alg}_{2}} \Rightarrow (-)^\circ \circ ((-)^{\circ})^{co},
\end{equation} 
whose associated 1- and 2-morphisms are invertible, each in the appropriate sense. More precisely, we need a family of invertible 1-morphisms $\et_A$, and a family of invertible 2-morphisms $\et_{M}$, such that $\et_A : A \rightarrow ((A)^{\circ})^\circ = A$, and such that the $\et_{M}$ witnesses the ``commutativity'' of the squares
\begin{equation}\label{nat square}
\begin{tikzcd}
A \arrow[d, "M"'] \arrow[rr,"\et_A"] & &  A \arrow[d, "M^{\circ \circ}"]	\\
B \arrow[rr,"\et_B"']	\arrow[shorten <=20pt,shorten >=20pt,urr,"\et_{M}", Rightarrow] & & B.
\end{tikzcd}
\end{equation}
For the 1-morphisms $\et_A$ we choose the identity bimodules ${}_{A}A_A$, which are clearly invertible. For the 2-morphisms we define $\et_{M}$ to be the isomorphism of bimodules given by 
\begin{equation}\label{eq:commcell}
{}_{A}(M \otimes_{B} B)_B \simeq {}_{A}M_B \simeq {}_{A}(\homm_{A}(\homm_B(M,B), A))_B \simeq {}_{A}(A\otimes_{A}\homm_{A}(\homm _B(M,B), A))_B,
\end{equation}
where the middle isomorphism is the canonical isomorphism ${}_{A}M_B  \rightarrow {}_{A}(M^{\circ \circ})_B$; see Appendix \ref{appendix:modules} for details.

For the invertible modification $\zeta$ we need to specify invertible 2-cells
\begin{equation}
\zeta_A : \et_{A^\circ} \Rightarrow (\et_A)^{\circ}.
\end{equation}
For fixed $A$, we choose as $\zeta_A$ the inverse of the isomorphism
\begin{equation}
(\et_A)^{\circ} = {}_{A^{op}}(\homm_A(A,A))_{A^{op}}  \rightarrow  {}_{A^{op}}{A^{op}}_{A^{op}} = \et_{A^\circ} 
\end{equation}
which already appeared as part of the coherence data for the bifunctor $(-)^{\circ}$, namely in (\ref{eq:coe1}). 

We can now state the following
\begin{thm}\label{circ is duality inv}
The bifunctor $(-)^{\circ}$ together with $\et$ and $\zeta$ defines a weak duality involution on $\mathrm{Alg}^{fd}_{2}$.
\end{thm}
We have deferred the proof of the above theorem to Appendix \ref{appendix:modules} is duality involution.  
\begin{rmk} The weak duality involution $(-)^{\circ}$ on $\Alg^{fd}$ can be regarded as an instance of the procedure outlined in \cite[Ex. 2.10]{shulman}. Our concrete description is needed in order to prove the main theorem in Section \ref{section:rep}.  
\end{rmk}
\begin{rmk}
We find it interesting to notice that the data needed to make $(-)^{\circ}$ into a duality involution is \emph{entirely} produced from the coherence data needed to define $\Alg$ and the pseudofunctor $(-)^{\circ}$ itself. A similar remark applies to the duality involution on $\KV$, though the coherence data in this case is trivial.
\end{rmk}
\section{Rep as a duality pseudofunctor}\label{section:rep}
In this section we show that the bifunctor $\Rep:\Alg^{fd}\to\KV$ introduced in Section \ref{section:finite} can be canonically equipped with the structure of a duality pseudofunctor.

According to Definition \ref{def:dualityfunctor}, we need to provide a pseudonatural adjoint equivalence $\ii$ and a modification $\theta$ satisfying a compatibility diagram.

\noindent\emph{Definition of $\ii$}: we need to specify a pseudonatural equivalence of the following form
\begin{equation}
\begin{tikzcd}
(\Alg^{fd})^{co} \arrow[r, "\text{Rep}^{co}"] \arrow[d, "(-)^{\circ}"'] & (\KV)^{co} \arrow[d, "(-)^{op}"] \arrow[shorten <=10pt,shorten >=10pt, dl, "\ii", Rightarrow] \\
\Alg^{fd} \arrow[r, "\Rep"']  & \KV .
\end{tikzcd}
\end{equation}
This consists of a family of invertible 1-morphisms in $\KV$
\begin{equation}
\ii_A : \text{Rep}(A)^{op} \longrightarrow \text{Rep}(A^\circ),\quad\forall{A}\in\Alg^{fd},
\end{equation}
and a family of invertible 2-morphisms
\begin{equation}
\begin{tikzcd}
(\Rep A)^{op} \arrow[r, "\ii_A"] \arrow[d, "((-) \otimes_A M)^{op}"'] & \Rep(A^\circ) \arrow[d, "(-) \otimes_{A^\circ} M^{\circ}"] \arrow[shorten <=10pt,shorten >=10pt, dl, "\ii_M", Rightarrow] \\
(\Rep B)^{op} \arrow[r,  "\i_B"']  & \Rep(B^\circ)
\end{tikzcd}
\end{equation}
for every bimodule ${}_{A}M_B$, satisfying the usual pseudonaturality conditions.

Define $\ii_A$ to be the additive functor which
\begin{itemize}
\item to any right $A$-module $V_A$ assigns the right $A^{op}$-module $V^{\circ}_{A^{\circ}}:=\homm_A(V,A)_{A^{op}}$ 
\item to any morphism $f^{op}:V_A \to W_A$ assigns ${f^{*}}: V^{\circ}_{A^{\circ}}\to W^{\circ}_{A^{\circ}}$. 
\end{itemize}
Note that $\i_{A}$ is an exact functor, i.e. a 1-morphism in $\KV$.

For any bimodule ${}_{A}M_B$, let $\ii_M$ be the natural isomorphism whose component at $V \in (\Rep A)^{op}$ is given by the canonical isomorphism 
\begin{equation}\label{eq:coherencecirc}
(\ii_{M})_{V}:V^\circ \otimes_{A^{op}} M^\circ \xrightarrow{\simeq} (V \otimes_A M)^\circ 
\end{equation}
obtained by combining the various theorems\footnote{Recall that all the bimodules we are considering are automatically projective as left and right modules.} in Appendix \ref{appendix:modules}. We leave to the reader to check that $\ii_{M}$ is indeed a natural isomorphism.
\begin{lemma} The family $\ii:=\left\{\ii_{A}, \ii_{M}\right\}_{A,M\in\Alg^{fd}}$ gives rise to a pseudonatural transformation.
\end{lemma}
To make $\ii$ into a pseudonatural adjoint equivalence, we show that $\ii$ is an equivalence, and invoke Theorem \label{thm:adjointequiv}, and the subsequent remark.

We define a (pseudo) inverse $\ii^\square$, whose component of 1-morphisms are 
\begin{equation}
\ii^\square_A := \ii_{A^\circ}^{op},
\end{equation}
and whose component 2-morphisms are
\begin{equation}
\ii^\square_M := \ii_{M^\circ}^{op}.
\end{equation}
To make $\ii$ and $\ii^\square$ into an equivalence pair, we consider as unit the invertible modification $\epsilon$ whose component at $A$ is the natural isomorphism 
\begin{equation}
\epsilon_{A}:1_{(\text{Rep}A)^{op}} \Rightarrow \ii^\square_A \circ \ii_A ,
\end{equation}
the component of which at $V \in (\Rep A)^{op}$ is the canonical isomorphism
\begin{equation}\label{eq: adj}
(\epsilon_{A})_{V}:V \longrightarrow (\ii^\square_A \circ \ii_A)(V)= \homm_{A^{\circ}}(\homm_A(V,A),{A^{\circ}}). 
\end{equation}
provided by Theorem \ref{thm:adj} in Appendix \ref{appendix:modules}. 

By Theorem \ref{thm:adjointequiv}, we can consider the unique adjoint equivalence in $\KV((\Rep^{co})^{\circ}),\Rep\circ(-)^{\circ})$ associated to $\ii$, $\ii^{\Box}$ and $\epsilon$.

\noindent\emph{Definition of $\theta$}: Now  we construct a modification $\theta$ whose components are invertible 2-morphisms in $\KV$ of the following form
\begin{equation}
\begin{tikzcd}
\Rep A \arrow[r, "(\ii_A)^{op}", {name = U}] & \Rep(A^\circ)^{op} \arrow[d, "\ii_{A^{op}}"]  \\
\Rep A\arrow[u, "\id_{\Rep A}"] \arrow[r,  "(-)\otimes_{A}A"', {name = D}]  & \Rep A \arrow[shorten <=10pt,shorten >=10pt, from = U, to = D, "\theta_{A}", Rightarrow]
\end{tikzcd}
,\quad \forall {A}\in\Alg^{fd}.
\end{equation}
Namely, $\theta_{A}$ is a natural isomorphism between $\ii_{A^{op}}\circ(\ii_{A})^{op}$ and $(-)\otimes_{A}A$. We choose its component at $V\in\Rep A$ to be the isomorphism
\begin{equation}
(\theta_{A})_{V}:V^{\circ\circ}\xrightarrow{\simeq} V\otimes_{A}A
\end{equation}
obtained as the following composition of canonical isomorphisms
\begin{equation}
V^{\circ\circ}\xrightarrow{\simeq}V\xrightarrow{\simeq}V\otimes_{A}A,
\end{equation}
where the first one is the inverse of the isomorphism in (\ref{eq: adj}).

We following is easily checked.
\begin{lemma}
The family $\theta:=\left\{\theta_{A}\right\}_{A\in\Alg^{fd}}$ defines a modification.
\end{lemma}
We can now prove our main theorem
\begin{thm}\label{thm:rep}The bifunctor $\Rep:\Alg^{fd}\to\KV$ equipped with the pseudonatural adjoint equivalence $\ii$ and the modification $\theta$ is a duality pseudofunctor.
\end{thm} 
\begin{proof}
We need to check that $\ii$ and $\theta$ satisfy the commutativity diagram\footnote{Notice that the diagram in \cite{shulman} contains a small typo.} in \cite{shulman}. Namely, we need to show that $\forall A\in\Alg^{fd}$ the 2-morphism
\begin{equation}\label{diag:LHS}
\begin{tikzcd}[column sep=huge]
(\Rep A)^{op} \arrow[r, "\ii_{A}"] & \Rep A^{op}\arrow[d, "(\ii_{A^{op}})^{op}"]\arrow[shorten <=20pt,shorten >=20pt, dl, "(\theta_{A}^{-1})^{op}", Rightarrow]\\
(\Rep A)^{op}\arrow[d, "\ii_{A}"']\arrow[u, bend left = 30, "\id" name = U]\arrow[u, bend right = 30, "\id"' name=D] \arrow[r, "((-)\otimes_{A}A))^{op}"'] & (\Rep A)^{op}\arrow[d, "\ii_{A}"] \arrow[shorten <=3pt,shorten >=3pt, Rightarrow, from = U,  to = D, "\id"]\arrow[shorten <=22pt,shorten >=22pt, dl, "\simeq", Rightarrow]\\
(\Rep A)^{op} \arrow[r, "(-)\otimes_{A^{op}}\homm_{A}{(A,A)}"'] & \Rep A^{op}
\end{tikzcd}
\end{equation}
from $\ii_{A}\circ(\ii_{A^{op}})^{op}\circ\ii_{A}$ to $(-)\otimes_{A^{op}}\homm_{A}{(A,A)}\circ\ii_{A}$ must coincide with the 2-morphism
\begin{equation}\label{diag:RHS}
\begin{tikzcd}[column sep=huge, row sep = huge]
(\Rep A)^{op} \arrow[shorten <=34pt,shorten >=34pt, dr, "\simeq", Rightarrow]\arrow[r, "\ii_{A}"] & \Rep A^{op} \arrow[r, "(\ii_{A^{op}})^{op}"] & (\Rep A)^{op}\arrow[d, "\ii_{A}"]\arrow[shorten <=30pt,shorten >=30pt, dl, "\theta_{{A}^{op}}"', Rightarrow]\\
(\Rep A)^{op} \arrow[u, "\id"]\arrow[r, "\ii_{A}"'] & \Rep A^{op}\arrow[u, "\id"'] \arrow[r, bend left = 13, "(-)\otimes_{A^{op}}A^{op}" name = U]\arrow[r, bend right = 13, "(-)\otimes_{A^{op}}\homm_{A}{(A,A)}"' name = D] & \Rep A^{op}\arrow[shorten <=4pt,shorten >=4pt, Rightarrow, from = U, to = D, "(\zeta_{A})_{*}"]
\end{tikzcd},
\end{equation}
where $(\zeta_{A})_*$ denotes the natural transformation induced by $\zeta_{A}$.

To help the reader in the pasting procedure, one can regard the diagram (\ref{diag:LHS}) to be of the following globular form
\begin{equation}
\begin{tikzcd}[column sep = huge, row sep = large]
\bullet \arrow[rr, bend right = 60, "(-)\otimes_{A^{op}}\homm_{A}{(A,A)}\circ\:\ii_{A}"']\arrow[r, bend left = 50, "(\ii_{A^{op}})^{op}\circ\:\ii_{A}" {name = U}]\arrow[r, bend right = 50, "" {name = D}] \arrow[shorten <=10pt,shorten >=10pt, Rightarrow, from = U, to = D, "(\theta_{A}^{-1})^{op}_{*}"']& \bullet \arrow[r, bend left = 60, "\ii_{A}" {name = P}]\arrow[r, bend right = 60, "" {name = Q}]  \arrow[shorten <=10pt,shorten >=10pt, Rightarrow, from = P, to = Q, "\id"']& \bullet \\
& & 
\end{tikzcd}
\end{equation}
while the diagram (\ref{diag:RHS}) has the following form
\begin{equation}
\begin{tikzcd}[column sep = huge, row sep = large]
\bullet \arrow[r, bend left = 60, "\ii_{A}" {name = P}]\arrow[r, bend right = 60, "\ii_{A}"' {name = Q}]\arrow[shorten <=13pt,shorten >=13pt, Rightarrow, from = P, to = Q, "\id"] & \bullet \arrow[r, bend left = 60, "\ii_{A}\circ\:(\ii_{A^{op}})^{op}" {name = U}] \arrow[r, "" {name = D}]\arrow[shorten <=6pt,shorten >=-1.0pt, Rightarrow, from = U, to = D, "\theta_{A^{op}}"]\arrow[r, bend right = 60, "(-)\otimes_{A^{op}}\homm_{A}{(A,A)}"' {name = E}]\arrow[shorten <=6pt,shorten >=6pt, Rightarrow, from = D, to = 
E, "(\zeta_{A})_{*}"]& \bullet
\end{tikzcd}
\end{equation}
For $V\in{(\Rep A)^{op}}$, the pasting of the diagram (\ref{diag:LHS}) gives rise to the following isomorphism
\begin{equation}\label{eq:morLHS}
V^{\circ\circ\circ}\xrightarrow{((\theta^{-1}_{A})_{V})^{\circ}}(V\otimes_{A}A)^{\circ}\xrightarrow{\simeq}V^{\circ}\otimes_{A^{op}}A^{\circ}
\end{equation}
where the second isomorphism is provided by the inverse of (\ref{eq:coherencecirc}).

On the other hand, the pasting of the diagram in (\ref{diag:RHS}) gives rise to the following isomorphism
\begin{equation}\label{eq:morRHS}
V^{\circ\circ\circ}\xrightarrow{(\theta_{A^{op}})_{V^{\circ}}}V^{\circ}
\otimes_{A^{op}}A^{op}\xrightarrow{\id\otimes\zeta_{A}}V^{\circ}\otimes_{A^{op}}A^{\circ}
\end{equation}
To see that (\ref{eq:morLHS}) and (\ref{eq:morRHS}) are equal, notice that the following diagram commutes
\begin{equation}
\begin{tikzcd}
V^{\circ\circ\circ}\arrow[r,"{((\theta_{A}^{-1})_{V})}^{\circ}"] \arrow[dr, "\psi^{-1}_{V^{\circ}}"']& (V\otimes_{A}A)^{\circ} \arrow[d,"1_{V}^{\circ}"]\\
& V^{\circ}
\end{tikzcd},
\end{equation}
where $1_{V}$ denotes the canonical isomorphism $V\to V\otimes_{A}A$, and $\psi_V$ denotes the isomorphism in Appendix \ref{appendix:modules}, Theorem \ref{thm:adj2}. This is due to the fact that $\psi_{V^{\circ}}=(\psi^{\circ}_{V})^{-1}$, and that by definition $(\theta_{A})_{V}=1_{V}\circ(\psi_{V})^{-1}$.
Similarly, the following diagram commutes
\begin{equation}
\begin{tikzcd}
{(V\otimes_{A}A)}^{\circ} \arrow[d,"1_{V}^{\circ}"] \arrow[r, "\simeq"] & V^{\circ}\otimes_{A^{op}}A^{\circ}\\
V^{\circ} \arrow[r, "1_{V^{\circ}}"'] & V^{\circ} \otimes_{A^{op}}A^{op} . \arrow[u, "\id\otimes\zeta_{A}"']
\end{tikzcd}
\end{equation}
This follows from the definition of the isomorphism in (\ref{eq:coherencecirc}). If we combine the two diagrams we obtain the following commutative diagram
\begin{equation}
\begin{tikzcd}
V^{\circ\circ\circ}\arrow[r,"{((\theta_{A}^{-1})_{V})}^{\circ}"] \arrow[dr, "\psi^{-1}_{V^{\circ}}"']& (V\otimes_{A}A)^{\circ} \arrow[d,"1_{V}^{\circ}"] \arrow[r, "\simeq"] & V^{\circ}\otimes_{A^{op}}A^{\circ}\\
& V^{\circ} \arrow[r, "1_{V^{\circ}}"'] & V^{\circ} \otimes_{A^{op}}A^{op} \arrow[u, "\id\otimes\zeta_{A}"']
\end{tikzcd}
\end{equation}
The upper composition corresponds to the isomorphism (\ref{eq:morLHS}), while the lower composition corresponds to the isomorphism (\ref{eq:morRHS}), after we notice that $1_{V^{\circ}}\circ\psi^{-1}_{V^{\circ}}=(\theta_{A^{op}})_{V^{\circ}}$. \end{proof}
\section{The general setting}\label{section:general}
In this section we briefly describe a general setting for the results discussed in the previous sections. We provide compact definitions of known concepts, and leave the full details of the various statements to future developments.

\subsection{Algebras in finite tensor categories and their Morita bicategory}\label{moritacat}
In the following $\calc$ denotes a symmetric semisimple finite tensor\footnote{We follow the convention in \cite{tenscat}, and assume that the category is rigid.} category. In other words, $\calc$ is a symmetric \emph{fusion category}; we refer to \cite{tenscat} for details concerning finite tensor categories and symmetric monoidal structures. The following definition is standard.
\begin{definition}
An algebra $A$ in $\calc$ is an object equipped with a multiplication $m:A\otimes{A}\to{A}$ and a unit $u:1\to{A}$ satisfying the appropriate commutative diagrams.
\end{definition}
Though an algebra is technically a triple $(A,m,u)$, we refer to $A$ as an algebra. A morphism of algebras is a morphism in $\calc$ which is compatible with the multiplication map and the unit in an obvious manner. 

Since $\calc$ is symmetric monoidal, we can define the opposite algebra $A^{op}$.
\begin{definition}
Let $A$ be an algebra in $\calc$. The \emph{oppositite algebra} $A^{op}$ is given by equipping $A$ with the following multiplication 
\begin{equation}
A\otimes{A}\xrightarrow{\sigma_{A,A}}A\otimes{A}\xrightarrow{m}A
\end{equation}
where $\sigma_{A,A}$ denotes the braiding isomorphism of $A$. 
\end{definition}
We moreover have the notion of a right $A$-module.
\begin{definition}
For an algebra $A$ in $\calc$, a \emph{right $A$-module} is an object $M$ in $\calc$ equipped with a morphism
\begin{equation}
M\otimes A\xrightarrow{\rho}M
\end{equation}
called a \emph{right action of $A$}, which satisfies appropriate commutative diagrams.
\end{definition}
A left $A$-module is defined analogously. Similar to the rest of the paper, when we want to emphasize that an object $M$ in $\calc$ is a right (resp. left) $A$-module, we use the notation $M_{A}$ (resp. $_{A}M$).

For $A$ and $B$ algebras in $\calc$, an $(A,B)$-bimodule $M$ is an object in $\calc$ which is a left $A$-module and a right $B$-module, and such that the two actions are compatible. We use $_{A}M_{B}$ to denote $(A,B)$-bimodules.

The following lemma is standard as well.
\begin{lemma}Let $(M,\rho)$ be a right $A$-module. Then the morphism
\begin{equation}
A\otimes{M}\xrightarrow{\sigma_{A,M}}M\otimes{A}\xrightarrow{\rho}M
\end{equation}
equips $M$ with the structure of a left $A^{op}$-module.
\end{lemma}
Similarly, any left $A$-module is canonically a right $A^{op}$-module.

A morphism between $A$-modules is naturally defined as a morphism in $\calc$ which is compatible with the action $\rho$. In particular, right (resp. left) $A$-modules form a $\fieldk$-linear category $\Mod_{A}$ (resp. $_{A}\Mod$). Moreover, both  $\Mod_{A}$ and $_{A}\Mod$ are $\fieldk$-linear abelian categories.
\begin{prop}\cite{tenscat}
Let $A$ be an algebra in a finite tensor category $\calc$. Then $\Mod_{A}$ is a finite category.
\end{prop}
The notion of tensor product of $A$-modules can be expressed in general terms.
\begin{definition}
Let $(M_{A},\rho_{M})$ and $(_{A}N, \rho_{N})$ be $A$-modules. The tensor product $M\otimes_{A}N$ is defined as the following coequalizer diagram
\begin{equation}
\begin{tikzcd}
M\otimes{A}\otimes{N}\arrow[r, shift left, "\rho_{M}\otimes\id_{N}"]\arrow[r, shift right, "\id_{M}\otimes\rho_{N}"'] & M\otimes{N}\arrow[r] & M\otimes_{A}N
\end{tikzcd}
\end{equation}
\end{definition}
Since $\calc$ is abelian, the coequalizer above is given by the cokernel of the morphism $\rho_{M}\otimes\id_{N} - \id_{M}\otimes\rho_{N}$. Hence tensor products of modules always exist.

One can show that for bimodules $_{A}M_{B}$ and $_{B}N_{C}$, the tensor product $M\otimes_{B}N$ carries canonically the structure of an $(A,C)$-bimodule, and that the usual canonical isomorphisms are guaranteed. Namely, we have that $(M{\otimes}_{B}N)\otimes_{C}P\simeq M{\otimes}_{B}(N\otimes_{C}P)$, and $A\otimes_{A}{M}\simeq{M}\simeq{M{\otimes_{B}}B}$. See \cite{tenscat} for details.

It is natural then to consider the following\footnote{Beware of the different notation as in \cite{lurieTQFT}!}
\begin{definition}
The Morita bicategory $\Alg(\calc)$ of algebras in $\calc$ is the bicategory where:
\begin{itemize}
 \item the objects are algebras in $\calc$; 
 \item the 1-morphisms are bimodules; and
 \item the 2-morphisms are morphisms between bimodules. 
\end{itemize}
Composition of 1-morphisms is given by tensoring of bimodules, and the unit 1-morphism for any algebra $A$ is given by $A$ itself regarded as an $(A,A)$-bimodule.
\end{definition}
Notice that since $\calc$ is symmetric monoidal, the tensor product $A\otimes{B}$ for algebras $A$ and $B$ in $\calc$ is canonically an algebra. One can indeed show that the tensor product in $\calc$ induces a symmetric monoidal structure on $\Alg(\calc)$. Moreover, every object $A$ in $\Alg(\calc)$ admits a dual object with respect to this monoidal structure, namely $A^{op}$. More precisely, we have the following
\begin{lemma}
Let $A$ be an algebra in $\calc$. Then its dual is given by the opposite algebra $A^{op}$, and as evaluation and coevaluation we can take $A$ regarded as an  $(A\otimes{A}^{op},1_{\calc})$-bimodule and an $(1_{\calc},A^{op}\otimes{A})$-bimodule, respectively.
\end{lemma}
In the lemma above, $1_{\calc}$ denotes the tensor unit in $\calc$.
We can then consider the fully dualisable subcategory $\Alg^{fd}(\calc)$ of $\Alg(\calc)$. 
\begin{definition}
An algebra $A$ in $\calc$ is called \emph{separable} if the multiplication morphism $m:A\otimes{A}\to A$ splits as a morphism of bimodules.
\end{definition}

\begin{prop}
An algebra $A$ in $\calc$ is fully dualisable if and only if it is separable.
\end{prop}
\begin{proof}
The proof is obtained by closely mimicing that in \cite{schompries}.
\end{proof}
\begin{rmk} For $\calc=\Vect_{\fieldk}$, we have that $\Alg(\calc)=\Alg$.
\end{rmk}
\begin{rmk}Notice that the ``finite-dimensionality'' condition on $A$ is subsumed by the fact that $\calc$ is rigid.
\end{rmk}
The objects of $\Alg^{fd}(\calc)$ are then the separable algebras in $\calc$, and the 1-morphisms are bimodules $_{A}M_{B}$ which admit right and left adjoints $(_{A}M_{B})^{\vee}$ and ${}^{\vee}(_{A}M_{B})$.

We can now consider the pseudofunctor\footnote{We work under the tacit assumption that right and left adjoints for 1-morphisms have been chosen.}
\begin{equation}
(-)^{\circ}:\Alg^{fd}(\calc)^{co}  \to \Alg^{fd}(\calc)
\end{equation}
defined as follows:
\begin{itemize}
\item to a separable algebra $A$ it assigns $A^{op}$;
\item to a bimodule $_{A}M_{B}$ it assigns  $_{A^{op}}{M^{\vee}}_{B^{op}}$; and
\item to $f^{op}:_{A}M_{B}\to{_{A}N_{B}}$ it assigns $f^{\vee}:_{A^{op}}{M^{\vee}}_{B^{op}}\to _{A^{op}}{N^{\vee}}_{B^{op}}$
\end{itemize}

Following the ideas and techniques discussed in the previous sections, we formulate the following
\begin{claim}
The bifunctor $(-)^{\circ}$ can be canonically made into a weak duality involution on $\Alg^{fd}(\calc)$.
\end{claim} 
\begin{rmk}Similar to Section \ref{section:alg}, the coherence data for $(-)^{\circ}$ arise from the universal properties of adjoints of 1-morphisms in a bicategory.
\end{rmk}
\subsection{Module categories} In this section we introduce a substitute for KV-vector spaces, in order to be able to construct a bifunctor $\Rep$ from $\Alg^{fd}(\calc)$. In the following, $\calc$ is a category satisfying the same assumptions as in Section \ref{moritacat}. Also here, our main references are \cite{tenscat,DSPS}.
\begin{definition}
A left $\calc$-module category is a locally finite abelian $\fieldk$-linear category $\calm$ equipped with a bilinear functor $\otimes^{\calm}:\calc\times\calm\to\calm$ together with isomorphisms witnessing the natural conditions for an action.
\end{definition}
A right $\calc$-module category can be similarly defined.
\begin{definition}
A left $\calc$-module functor between left $\calc$-module categories $\calm$ and $\caln$ is a linear functor $\calf:\calm\to\caln$ together with isomorphisms $f_{x,m}:\calf(x\otimes{m})\simeq x \otimes\calf(m)$ satisfying the appropriate pentagon and triangle relations. 
\end{definition}
\begin{definition}
A left $\calc$-module natural transformation between left $\calc$-module functors $\calf$ and $\calg$ is a natural transformation $\eta:\calf\to\calg$ satisfying the condition $(id_{x}\otimes\eta_{m})\circ f_{x,m}=g_{x,m}\circ\eta_{x\otimes{m}}$.
\end{definition}
Right $\calc$-module functors and natural transformations can be defined similarly.

Left $\calc$-module categories together with left exact $\calc$-module functors and $\calc$-module natural transformations form a 2-category $\Mod(\calc)$.
\begin{ex}
Let $A$ be an algebra in $\calc$. Then $\Mod_{A}$ is canonically a left $\calc$-module category via the functor
\begin{align}
\begin{split}
\calc\times\Mod_{A} & \to\Mod_{A}\\
(x,m)& \to x\otimes{m}
\end{split}
\end{align}
\end{ex}
\begin{rmk}
In \cite{kv}, 2-vector spaces were introduced as module categories over $\Vectt_{\fieldk}$ with additional properties. 
\end{rmk}
For $\calm$ a left $\calc$-module category, let $\Endo_{l}(\calm)$ denote the $\fieldk$-linear monoidal category of left exact $\calc$-module functors from $\calm$ to $\calm$. A useful result concerning $\calc$-module categories is the following \cite{tenscat}
\begin{thm}\label{thm:monstruct} There is a bijection between structures of a left $\calc$-module category on $\calm$ and $\fieldk$-linear monoidal functors $\calc\to\Endo_{l}(\calm)$.
\end{thm}
In the following, we  assume that all our module categories $\calm$ are semi-simple as abelian categories. This is done in view of the following 
\begin{prop}Let $\calm$ be a $\calc$-module category which is also semi-simple. Then any left exact $\calc$-module functor $\calf:\calm\to\calm$ is exact. 
\end{prop}
For $\calm$ a semi-simple $\calc$-module category, we denote with $\Endo(\calm)$ the monoidal category of exact functors.
\begin{lemma}
For $\calm$ a semisimple $\calc$-module category, $\Endo(\calm)$ is a tensor category, where duals are given by adjoints.
\end{lemma}
Let $\calm$ be a left (semi-simple) $\calc$-module category, and consider the following composition of monoidal functors
\begin{equation}\label{eq:oppositemodule}
\calc\to\Endo(\calm)\xrightarrow{(-)^{R}}\Endo(\calm)^{mp}\simeq\Endo(\calm^{op})^{rev}
\end{equation}
where the first functor is the one given by Theorem \ref{thm:monstruct}, and
where $(-)^{R}$, $(-)^{rev}$ and $(-)^{mp}$ denote taking the right adjoint, taking the monoidally opposite category, and taking the monoidally opposite opposite category, respectively.

The monoidal functor in (\ref{eq:oppositemodule}) canonically provides a monoidal functor $\calc^{rev}\to\Endo(\calm^{op})$, and consequently\footnote{Recall that $\calc$ is symmetric monoidal, hence $\calc^{rev}\simeq\calc$.} a monoidal functor $\calc\to\Endo(\calm^{op})$. In other words, the composition above defines a left $\calc$-module structure on $\calm^{op}$. For notational clarity we denote by $\calm^{\circ}$ the $\calc$-module category $\calm^{op}$ equipped with the module structure above. Notice that we have a canonical identification $\calm^{\circ\circ}\simeq\calm$ as left $\calc$-module categories\footnote{This is essentially due to the fact that for any pair of functors $F$ and $G$ between categories, $F\dashv{G}$ implies $G^{op}\dashv F^{op}$.}.
\begin{rmk}
Any category $\calm$ enriched over $\calc$ as above can be canonically given the structure of a left $\calc$-module structure. Then $\calm^{\circ}$ is the left $\calc$-module category corresponding to the opposite of $\calm$ as a $\calc$-enriched category\footnote{Note that the opposite of an enriched category can be defined only if the enriching category is symmetric monoidal.}.  
\end{rmk}
One can argue straightforwardly that for any (exact) $\calc$-module functor $\calf:\calm\to\caln$, the opposite functor $\calf^{op}$ can be given the structure of a $\calc$-module functor $\calf^{\circ}$ between $\calm^{\circ}$ and $\caln^{\circ}$.  The story is similar for natural transformations.

Let $\Mod^{ss}(\calc)$ denote the 2-category of semi-simple left $\calc$-module categories, exact $\calc$-module functors and $\calc$-module natural transformations.

Consider the pseudofunctor
\begin{equation}
(-)^{\circ}:\Mod^{ss}(\calc)^{co} \to \Mod^{ss}(\calc)
\end{equation}
defined as follows:
\begin{itemize}
\item to a module category $\calm$ it assigns $\calm^{\circ}$;
\item to a module functor $\calf:\calm\to\caln$ it assigns $\calf^{\circ}:\calm^{\circ}\to\caln^{\circ}$; and
\item to $\eta^{op}:\calf\to\calg$ it assigns $\eta^{\circ}:\calf^{\circ}\to\calg^{\circ}$
\end{itemize}
It is reasonable to expect then the following
\begin{claim} The bifunctor $(-)^{\circ}$ can be canonically made into a weak duality involution on $\Mod^{ss}(\calc)$.
\end{claim}
\subsection{Representations} Similar to what we have done in the previous sections of this paper, we can connect the bicategory $\Alg^{fd}(\calc)$ to $\Mod^{ss}(\calc)$ via the bifunctor $\Rep^{\calc}$ given by taking modules over algebras. To this aim, we can use the following results \cite{tenscat}
\begin{prop}Let $A$ be a separable algebra in a fusion category $\calc$. Then $\Mod_{A}$ is a semi-simple left $\calc$-module category.
\end{prop}
\begin{prop}Let $A$ and $B$ be algebras in $\calc$, and let $_{A}M_{B}$ be an $(A,B)$-bimodule. Then the functor
\begin{equation}
(-)\otimes_{A}M:\Mod_{A}\to\Mod_{B}
\end{equation}
is a right exact $\calc$-module functor.
\end{prop}
We can now consider the following pseudofunctor
\begin{equation}
\Rep^{\calc}:\Alg^{fd}(\calc) \to \Mod^{ss}(\calc)
\end{equation}
defined as follows:
\begin{itemize}
\item to a separable algebra $A$ it assigns the semi-simple $\calc$-module category $\Mod_{A}$;
\item to a bimodule $_{A}M_{B}$ it assigns $(-)\otimes_{A}M:\Mod_{A}\to\Mod_{B}$; and
\item to a morphism $f:M\to N$ it assigns the associated natural transformation $(-)\otimes_{A}M \Rightarrow (-)\otimes_{A}N$.
\end{itemize}
The fact that the above is a pseudofunctor is a corollary of the properties of algebra bimodules and their tensor product. Indeed, the coherence data can be defined as in Section \ref{section:alg}.

We conclude the paper with the statement of a result that can be straightforwardly obtained following the lines of the special case proven in Section \ref{section:rep}.
\begin{claim} The bifunctor $\Rep^{\calc}$ can be canonically equipped with the data of a duality pseudofunctor between $\Alg^{fd}(\calc)$ and $\Mod^{ss}(\calc)$ equipped with their respective weak duality involutions.
\end{claim}
\appendix
\section*{Appendix}
\renewcommand{\thesection}{A} 

\subsection{Background on modules over finite-dimensional algebras}\label{appendix:modules}
In the following, we recall the basic material we need regarding finite-dimensional modules over finite-dimensional $\fieldk$-algebras. We fix a field $\fieldk$ which is of characteristic 0 and algebraically closed. We will mainly follow \cite{frobbook}, to which we refer the reader for the proofs of the various statements.

Let $A$ be a finite-dimensional $\fieldk$-algebra. Recall that the category $\Mod_{A}$ of finite-dimensional right modules over $A$ is an abelian category. Recall also that for any right $A$-module $M$, the vector space $\hom_{A}(M,A)$ comes equipped canonically with a left $A$-module structure induced by left multiplication on $A$.
\begin{definition}
An object $P \in \Mod_{A}$ is called \emph{projective} if the functor $\hom_{A}(P,-):\Mod_{A}\to\Vectt_{\fieldk}$ is exact.
\end{definition}
\begin{thm} Let $A$ be a \emph{semisimple} finite-dimensional $\fieldk$-algebra. Then any finite-dimensional module over $A$ is projective. 
\end{thm}
We recall also the notion of tensor product over an algebra
\begin{definition}
Let $M$ and $N$ be a right and left $A$-module, respectively. The \emph{tensor product over $A$} of $M$ and $N$ is the vector space given by
\begin{equation}
M\otimes_{A}N:=\left\{x\otimes{y}\mid x\in{M}, y\in{N} \right \}/\left\{xa\otimes y - x\otimes{ay}\right\}
\end{equation}
\end{definition}
The following lemma is immediate
\begin{lemma}\label{lemma:braid}
Let $M$ and $N$ be a right and left $A$-module, respectively. The canonical braiding on $\Vectt_{\fieldk}$ induces a linear isomorphism
\begin{equation}
M\otimes_{A}N\simeq{ N\otimes_{A^{op}}M}
\end{equation}
\end{lemma}

Notice that if $M$ is a $(B,A)$-bimodule and $N$ is a $(A,C)$-bimodule, then $M\otimes_{A}N$ canonically inehrits a $(B,C)$-bimodule structure. Moreover, the isomorphism in Lemma \ref{lemma:braid} is compatible with such bimodule structure.

The following theorem, called the \emph{adjoint theorem}, asserts that for any $(A,B)$-bimodule, the functors $(-)\otimes_{A}M$ and $\hom_{B}(M,-)$ form an adjoint pair
\begin{thm}\label{thm:adj} Let $A$ and $B$ be $\fieldk$-algebras, and let $M$ be a $(A,B)$-bimodule. Then for any right $A$-module $X$ and right $B$-module $Y$ the linear map
\begin{align}
\begin{split}
\hom_{B}(X\otimes_{A}M, Y) &\longrightarrow \hom_{A}(X,\hom_{B}(M,Y))\\
g & \longmapsto \left(x \longmapsto f_{x}:m \longmapsto g(x\otimes{m})\right)
\end{split}
\end{align}
is an isomorphism.
\end{thm}
In the case in which $Y$ is a $(C,B)$-bimodule, the vector spaces $\hom_{B}(X\otimes_{A}M, Y)$ and $\hom_{A}(X,\hom_{B}(M,Y))$ aquire a canonical structure of left $C$-module, induced by the left $C$-action on $Y$. It is routine to show that the isomorphism in Theorem \ref{thm:adj} is $C$-linear. Similarly if $X$ is a $(C,A$)-module.
\begin{thm}\label{thm:adj2}Let $A$ be a $\fieldk$-algebra, and let $P$ be a projective right $A$-module. Then:
\begin{itemize}
\item$\hom_{A}(P,A)$ is a projective left $A$-module; and
\item the linear map
\begin{align}
\begin{split}
\psi_{P}:P &\to \hom_{A^{op}}(\hom_{A}(P,A),A^{op})\\
p &\longmapsto (\psi_{P}(p): g \longmapsto g(p)) 
\end{split}
\end{align}
is an isomorphism of right $A$-modules.
\end{itemize}
\end{thm}
Notice that in the above theorem, we regard a left (right) $A$-module as a right (left) $A^{op}$-module. Similar to the previous theorem, in the case in which $P$ is a $(B,A)$-bimodule, it is routine to check that the isomorphism in Theorem \ref{thm:adj2} is $B$-linear. Moreover, $\psi_{P}$ is natural in $P$.
\begin{thm}\label{thm:adj3}
Let $A$ and $B$ be $\fieldk$-algebras, and let $P$ be a $(B,A)$-bimodule which is projective as a right $A$-module. Then for any right $A$-module $X$ the linear map
\begin{align}
\begin{split}
X\otimes_{A}\hom_{A}(P,A)&\to \hom_{A}(P,X)\\
x\otimes{g} &\longmapsto (p \longmapsto x\cdot g(p)) 
\end{split}
\end{align}
is an isomorphism of right $B$-modules and natural in $X$.
\end{thm}
Again, if $X$ is a $(C,A)$-bimodule, it is routine to check that the isomorphism in Theorem \ref{thm:adj3} is $C$-linear.
\subsection{Proof of Theorem \ref{circ is duality inv}}\label{appendix:proof}
We need to verify that $\zeta$ satisfies the compatibility required for a weak duality involution as stated in \cite{shulman}. Namely, we need to show that for any $A\in\Alg^{fd}$ we have the following equality of 2-morphisms
\begin{equation}\label{diagram: duality}
\begin{tikzcd}
A\arrow[r, "A"] & A \arrow[r, bend left=40, "A" {name=U}] \arrow[r, bend right=40, "\homm_{A^{op}}{(A^{op},A^{op})}"' {name=D}] & A \arrow[shorten <=5pt,shorten >=5pt, Rightarrow, from=U, to=D, "\zeta_{A^{op}}"]
\end{tikzcd}
=
\begin{tikzcd}
& A \arrow[d, shorten <=5pt,shorten >=5pt, Rightarrow, "\simeq"] \arrow[dr, bend left = 30, "A"]&\\
A\arrow[ur, "A"]\arrow[r, "A"'] & A \arrow[r, bend left=40, "A^{\circ\circ}" {name=U}] \arrow[r, bend right=40, "\homm_{A^{op}}{(A^{op},A^{op})}"' {name=D}] & A \arrow[shorten <=5pt,shorten >=5pt, Rightarrow, from=U, to=D, "\zeta^{\circ}_{A}"]
\end{tikzcd}
\end{equation}
First, recall that by construction 
\begin{align}
\begin{split}
\zeta_{A^{op}}:A & \to \homm_{A^{op}}(A^{op},A^{op})\\
x & \longmapsto \phi_{x}: a \longmapsto a\cdot{x}
\end{split}
\end{align}
and similarly
\begin{align}
\begin{split}
\zeta_{A}:A^{op} & \to \homm_{A}(A,A)\\
x & \longmapsto \bar{\phi}_{x}: a \longmapsto x\cdot{a}
\end{split}
\end{align}
The LHS of (\ref{diagram: duality}) is then given by the isomorphism
\begin{align}
\begin{split}
\id_{A}\otimes\zeta_{A^{op}}: A\otimes_{A}A & \to A\otimes_{A}\homm_{A^{op}}(A^{op},A^{op})\\
a\otimes b & \longmapsto a\otimes\phi_{b}.
\end{split}
\end{align}
\begin{rmk}
In the definitions above, the multiplication is \emph{always} performed in $A$.
\end{rmk}
On the other hand, the RHS of (\ref{diagram: duality}) is the following composition
\begin{equation}\label{eq:rhs}
A\otimes_{A}A\xrightarrow{\et_{A}}A\otimes_{A}A^{\circ\circ}
\xrightarrow{\id_{A}\otimes(\zeta_{A})^{*}}A\otimes_{A}\homm_{A^{op}}(A^{op},A^{op})
\end{equation}
where $\et_{A}$ is given by the 2-morphism defined in (\ref{eq:commcell}), namely it is given by the following composition
\begin{align}
\begin{split}
A\otimes_{A}A\xrightarrow{\simeq}A\xrightarrow{\simeq}A^{\circ\circ}\xrightarrow{\simeq}A
\otimes_{A}A^{\circ\circ}\\
a\otimes{b}\longmapsto a\cdot b \longmapsto f_{ab} \longmapsto 1\otimes f_{ab}
\end{split}
\end{align}
where 
\begin{equation}
f_{ab}(g):=g(ab),\quad\forall g\in\hom_{A}(A,A).
\end{equation}
Note now that $\forall x \in {A^{op}}$ we have the following
\begin{align}
\begin{split}
(\zeta_{A})^{*}(f_{ab})(x) &=f_{ab}(\zeta_{A}(x))\\
&=f_{ab}(\bar{\phi}_{x})\\
&=\bar{\phi}_{x}(a\cdot{b})\\
&=x\cdot a \cdot b\\
&=\phi_{b}(x\cdot{a})=\phi_{b}(a\star{x})\\
&=(\phi_{b}\star{a})(x)\\
&=(a\cdot\phi_{b})(x),
\end{split}
\end{align}
where for clarity we use $\cdot$ to denote an $A$-action, and $\star$ to denote an $A^{op}$-action. Hence the isomorphism in (\ref{eq:rhs}) is explicitely given by
\begin{equation}
a\otimes{b} \longmapsto 1\otimes a\cdot\phi_{b} = a\otimes\phi_{b},\quad \forall a,b\in{A}
\end{equation}
which agrees with the LHS in (\ref{diagram: duality}). $\Box$
\bibliography{bib}

\providecommand{\bysame}{\leavevmode\hbox to3em{\hrulefill}\thinspace}
\providecommand{\MR}{\relax\ifhmode\unskip\space\fi MR }
\providecommand{\MRhref}[2]{%
  \href{http://www.ams.org/mathscinet-getitem?mr=#1}{#2}
}
\providecommand{\href}[2]{#2}
\begin{thebibliography}{10}

\bibitem{bartlettschom1}
B.~{Bartlett}, C.~L. {Douglas}, C.~J. {Schommer-Pries}, and J.~{Vicary},
  \emph{{Extended 3-dimensional {B}ordism as the {T}heory of {M}odular
  {O}bjects}}, arXiv:1411.0945 (2014).

\bibitem{bartlettschom2}
\bysame, \emph{{Modular {C}ategories as {R}epresentations of the 3-dimensional
  {B}ordism 2-category}}, arXiv:1509.06811 (2015).

\bibitem{benabou}
J.~B\'{e}nabou, \emph{Introduction to bicategories}, Reports of the {M}idwest
  {C}ategory {S}eminar, Springer, Berlin, 1967, pp.~1--77. \MR{0220789}

\bibitem{davydovich}
O.~Davydovich, \emph{State {S}ums in {T}wo {D}imensional {F}ully {E}xtended
  {T}opological {F}ield {T}heories}, Ph.D. thesis, Univ. Texas, Austin, 2011.

\bibitem{DSPS}
Christopher~L. {Douglas}, Christopher {Schommer-Pries}, and Noah {Snyder},
  \emph{{The balanced tensor product of module categories}}, Kyoto J. Math.,
  Advance publication (2019).

\bibitem{tenscat}
P.~Etingof, S.~Gelaki, D.~Nikshych, and V.~Ostrik, \emph{Tensor {C}ategories},
  vol. 205, AMS, 2015.

\bibitem{finitetensor}
P.~Etingof and V.~Ostrik, \emph{Finite tensor categories}, Mosc. Math. J.
  \textbf{4} (2004).

\bibitem{gurski}
N.~Gurski, \emph{Biequivalences in tricategories}, Theory Appl. Categ.
  \textbf{26} (2012), 349--384. \MR{2972968}

\bibitem{gwillscheim}
O.~{Gwilliam} and C.~{Scheimbauer}, \emph{{Duals and adjoints in higher Morita
  categories}}, arXiv:1804.10924 (2018).

\bibitem{haugseng}
R.~Haugseng, \emph{The higher {M}orita category of {$\Bbb{E}_n$}-algebras},
  Geom. Topol. \textbf{21} (2017), 1631--1730. \MR{3650080}

\bibitem{kv}
M.~M. Kapranov and V.~A. Voevodsky, \emph{{$2$}-categories and {Z}amolodchikov
  {T}etrahedra {E}quations}, Algebraic groups and their generalizations:
  quantum and infinite-dimensional methods ({U}niversity {P}ark, {PA}, 1991),
  Proc. Sympos. Pure Math., vol.~56, Amer. Math. Soc., Providence, RI, 1994,
  pp.~177--259. \MR{1278735}

\bibitem{leinsterbicat}
T.~{Leinster}, \emph{Basic bicategories}, arXiv:math/9810017 (1998).

\bibitem{lurieTQFT}
J.~Lurie, \emph{{On the {C}lassification of {T}opological {F}ield {T}heories}},
  Current Developments in Mathematics \textbf{2008} (2009), 129--280.

\bibitem{lurieHigher}
\bysame, \emph{{Higher {A}lgebra}}, 2017,
  \url{http://www.math.harvard.edu/~lurie/papers/HA.pdf}.

\bibitem{morita}
K.~Morita, \emph{Duality for modules and its applications to the theory of
  rings with minimum condition}, Sci. Rep. Tokyo Kyoiku Daigaku Sect. A
  \textbf{6} (1958), 83--142. \MR{0096700}

\bibitem{schompries}
C.~Schommer{-}Pries, \emph{The {C}lassification of {T}wo-{D}imensional
  {E}xtended {T}opological {F}ield {T}heories}, arXiv:1112.1000 (2011).

\bibitem{shulman}
M.~Shulman, \emph{Contravariance through enrichment}, Theory Appl. Categ.
  \textbf{33} (2018), 95--130. \MR{3756532}

\bibitem{frobbook}
A.~Skowro\'{n}ski and K.~Yamagata, \emph{Frobenius algebras. {I}}, EMS
  Textbooks in Mathematics, European Mathematical Society (EMS), Z\"{u}rich,
  2011. \MR{2894798}

\end{thebibliography}
\bibliographystyle{amsplain}

\end{document}